\documentclass[12pt]{amsart}%
\usepackage{graphicx}
\usepackage{amscd}
\usepackage{geometry}
\usepackage{amsmath}
\usepackage{amsfonts}
\usepackage{amssymb}%
\providecommand{\U}[1]{\protect\rule{.1in}{.1in}}
\newtheorem{theorem}{Theorem}[section]
\theoremstyle{plain}

\newtheorem{corollary}[theorem]{Corollary}

\newtheorem{lemma}[theorem]{Lemma}

\newtheorem{proposition}[theorem]{Proposition}

\theoremstyle{definition}
\newtheorem{definition}[theorem]{Definition}
\newtheorem{remark}{Remark}
\newtheorem{example}{Example}
\numberwithin{equation}{section}

\begin{document}
\title{Boundaries of Baumslag-Solitar Groups}
\author{Craig R. Guilbault}
\address{Department of Mathematical Sciences\\
University of Wisconsin-Milwaukee, Milwaukee, WI 53201}
\email{craigg@uwm.edu}
\author{Molly A. Moran}
\address{Department of Mathematics, The Colorado College, Colorado Springs, Colorado 80903}
\email{mmoran@coloradocollege.edu}
\author{Carrie J. Tirel}
\address{Department of Mathematics\\
University of Wisconsin-Fox Valley, Menasha, WI 54952}
\email{carrie.tirel@uwc.edu}
\thanks{This research was supported in part by Simons Foundation Grants 207264 and
427244, CRG}
\date{August 16, 2018}

\begin{abstract}
A $\mathcal{Z}$-structure on a group $G$ was introduced by Bestvina in order
to extend the notion of a group boundary beyond the realm of CAT(0) and
hyperbolic groups. A refinement of this notion, introduced by Farrell and
Lafont, includes a $G$-equivariance requirement, and is known as an
$\mathcal{EZ}$-structure. The general questions of which groups admit
$\mathcal{Z}$- or $\mathcal{EZ}$-structures remain open. In this paper we add
to the current knowledge by showing that all Baumslag-Solitar groups admit
$\mathcal{EZ}$-structures and all generalized Baumslag-Solitar groups admit
$\mathcal{Z}$-structures.

\end{abstract}
\maketitle

\section{Introduction\label{Section: Introduction}}

In \cite{Bes96}, Bestvina introduced the concept of a $\mathcal{Z}$-structure
on a group $G$ to provide an axiomatic treatment of group boundaries. Roughly
speaking, the definition requires $G$ to act geometrically (properly,
cocompactly, by isometries) on a \textquotedblleft nice\textquotedblright%
\ space $X$ and for that space to admit a nice compactification $\overline{X}$
(a $\mathcal{Z}$-compactification). In addition, it is required that
translates of compact subsets of $X$ get small in $\overline{X}$ --- a property
called the \emph{nullity condition}. Adding visual boundaries to CAT(0) spaces
and Gromov boundaries to appropriately chosen Rips complexes provide the model
examples. To admit a $\mathcal{Z}$-structure, it is necessary that a group $G$ admits a finite $K(G,1)$ complex (a \emph{Type F} group).  Bestvina posed the still open question as to whether or not every Type F group admits a $\mathcal{Z}$-structure.

In \cite{Bes96}, the Baumslag-Solitar group $BS\left(  1,2\right)  $ was put
forward as a non-hyperbolic, non-CAT(0) group that, nevertheless, admits a
$\mathcal{Z}$-structure. Baumslag-Solitar groups $BS\left(  1,n\right)  $
behave similarly, but from the beginning, the status of general
Baumslag-Solitar groups $BS\left(  m,n\right)  $ was unclear. In this paper we
resolve that issue in a strong way.
%The nullity condition is often the roadblock for many groups and thus the list of known examples is relatively short.

\begin{theorem}
Every generalized Baumslag-Solitar group admits a $\mathcal{Z}$-structure.
\end{theorem}

A generalized Baumslag-Solitar group is the fundamental group of a graph of
groups with vertex and edge groups $\mathbb{Z}$. By applying work of Whyte
\cite{Why01} and a boundary swapping trick (see \cite{Bes96} and
\cite{GuMo18}), it will suffice to show that the actual Baumslag-Solitar
groups $BS(m,n)$ admit $\mathcal{Z}$-structures. For those groups, we will
prove the following stronger theorem.

\begin{theorem}
[$\mathcal{EZ}$-Structures on Baumslag-Solitar Groups]%
\label{Theorem: Boundary BS(m,n)} All Baumslag-Solitar groups, $BS(m,n)$,
admit $\mathcal{EZ}$-structures.
\end{theorem}

Here $\mathcal{EZ}$ stands for \textquotedblleft equivariant $\mathcal{Z}%
$-structure\textquotedblright, a $\mathcal{Z}$-structure in which the group
action extends to the boundary. Torsion-free groups (which includes all groups
studied in this paper) that admit $\mathcal{EZ}$-structures are known to
satisfy the Novikov conjecture (\cite{FaLa05}). That is one reason to aim for
this stronger condition.

\section{Background \label{Section: Background}}

\subsection{Visual Boundaries of CAT(0) Spaces}

In this section, we review the definition of CAT(0) spaces and the visual
boundary as we will use these as a starting point for $\mathcal{EZ}%
$-structures on $BS(m,n)$. For a more thorough treatment of CAT(0) spaces, see
\cite{BrHa99}.

\begin{definition}
A geodesic metric space $(X,d)$ is a \emph{CAT(0) space} if all of its
geodesic triangles are no fatter than their corresponding Euclidean comparison
triangles. That is, if $\Delta(p,q,r)$ is any geodesic triangle in $X$ and
$\overline{\Delta}(\overline{p},\overline{q},\overline{r})$ is its comparison
triangle in $\mathbb{E}^{2}$, then for any $x,y\in\Delta$ and the comparison
points $\overline{x},\overline{y}\in \overline{\Delta}$, then $d(x,y)\leq d_{\mathbb{E}}%
(\overline{x},\overline{y})$.
\end{definition}

\begin{example}
Basic examples of CAT(0) spaces include:

\begin{itemize}
\item $%
%TCIMACRO{\U{211d} }%
%BeginExpansion
\mathbb{R}
%EndExpansion
^{n}$ equipped with the Euclidean metric is a CAT(0) space as all geodesic
triangles are already Euclidean and hence no fatter than their comparison triangles.

\item A tree, $T$, is a CAT(0) space since all geodesic triangles are
degenerate and thus have no thickness associated to them.

\item If $X$ and $Y$ are CAT(0), spaces, then $X\times Y$ with the $\ell^{2}$
metric is CAT(0). So, for example, $%
%TCIMACRO{\U{211d} }%
%BeginExpansion
\mathbb{R}
%EndExpansion
\mathbb{\times}T$ is a CAT(0) space---a fact that will play a significant role
in this paper.
\end{itemize}
\end{example}

A group $G$ that acts properly, cocompactly, and by isometries (also known as
a \emph{geometric} group action) on a proper CAT(0) space is called a
\emph{CAT(0) group}.

\begin{definition}
The \emph{boundary} of a proper CAT(0) space $X$, denoted $\partial X$, is the
set of equivalence classes of rays, where two rays are equivalent if and only
if they are asymptotic. We say that two geodesic rays $\alpha, \alpha^{\prime
}:[0,\infty)\to X$ are \emph{asymptotic} if there is some constant $k$ such
that $d(\alpha(t),\alpha^{\prime}(t))\leq k$ for every $t\geq0$.
\end{definition}

If we fix a base point $x_{0}\in X$, each equivalence class of rays in $X$
contains exactly one representative emanating from $x_{0}$. So when $x_{0}$ is
chosen, we can view $\partial X$ as the set of all rays in $X$ based at
$x_{0}$. We may endow $\overline{X}=X\cup\partial X$, with the \emph{cone
topology}, described below, under which $\partial X$ is a closed subspace of
$\overline{X}$ and $\overline{X}$ compact (provided $X$ is proper). Equipped
with the topology induced by the cone topology on $\overline{X}$, the boundary
is called the \emph{visual boundary} of $X$; we will denote it by
$\partial_{\infty}X$.

%In what follows, the term `boundary' will always mean `visual boundary'. Furthermore, we will slightly abuse terminology and call the cone topology restricted to $\partial X$ simply the cone topology if it is clear that we are only interested in the topology on $\partial X$.

The cone topology on $\overline{X}$, denoted $\mathcal{T}(x_{0})$ for
$x_{0}\in X$, is generated by the basis $\mathcal{B}=\mathcal{B}_{0}%
\cup\mathcal{B}_{\infty}$ where $\mathcal{B}_{0}$ consists of all open balls
$B(x,r)\subset X$ and $\mathcal{B}_{\infty}$ is the collection of all sets of
the form
\[
U(c,r,\epsilon)=\{x\in\overline{X}\mid d(x,c(0))>r\text{ and }d(p_{r}%
(x),c(r))<\epsilon\}
\]
where $c:[0,\infty)\rightarrow X$ is any geodesic ray based at $x_{0}$, $r>0$,
$\epsilon>0$, and $p_{r}$ is the natural projection of $\overline{X}$ onto
$\overline{B}(c(0),r)$.

\begin{example}
\label{Example: CAT(0) boundaries}Boundaries of the simple examples given
above are:

\begin{itemize}
\item $\partial_{\infty}%
%TCIMACRO{\U{211d} }%
%BeginExpansion
\mathbb{R}
%EndExpansion
^{n}\simeq S^{n-1}$

\item $\partial_{\infty}T$ is compact and $0$-dimensional. If each vertex has
degree $\geq3$, it is a Cantor set $C$. (In order for $T$ to be proper, assume
all vertices have finite degree.)

\item If $X$ and $Y$ are CAT(0), spaces and $X\times Y$ is given the $\ell
^{2}$ metric, then $\partial_{\infty}(X\times Y)\simeq\partial X\ast\partial
Y$, the (spherical) join of the two boundaries. For example, $\partial
_{\infty}(%
%TCIMACRO{\U{211d} }%
%BeginExpansion
\mathbb{R}
%EndExpansion
\mathbb{\times}T)$ is homeomorphic to $S^{0}\ast\partial_{\infty}T$; the
suspension of a $0$-dimensional set (usually a Cantor set)..
\end{itemize}
\end{example}

When $G$ is a CAT(0) group acting geometrically on a proper CAT(0) space $X$,
we call $\partial_{\infty}X$ a CAT(0) boundary for $G$. For example, since
$\mathbb{Z}^{n}$ acts geometrically on $%
%TCIMACRO{\U{211d} }%
%BeginExpansion
\mathbb{R}
%EndExpansion
^{n}$, it is a CAT(0) group and $S^{n-1}$ is a CAT(0) boundary . The free
group on two generators, $F_{2}$, acts geometrically on a four-valent tree, so
a CAT(0) boundary for $F_{2}$ is the Cantor set.

The following lemma, which is reminiscent of the Lebesgue covering lemma, will
be useful in proving our main theorem.

\begin{lemma}
\label{Lemma: Cover of Boundary} Let $(X,d)$ be a proper CAT(0) space and let
$\mathcal{U}$ be an open cover of $\overline{X}$. Then there exists a
$\delta>0$ so that for every $z\in\partial X$, $U(z, \frac{1}{\delta},
\delta)$ lies in an element of $\mathcal{U}$.
\end{lemma}

\begin{proof}
Since $\partial X$ is compact, there is a finite subcollection $\{U_{1},
U_{2},..., U_{k}\}$ of $\mathcal{U}$ that covers $\partial X$. For each
$i\in\{1,2,...,k\}$, define a function $\eta_{i}:\partial X\to[0,\infty)$ by
$\eta_{i}(z)=\text{sup}\{\epsilon\, |\,U(z,\frac{1}{\epsilon}, \epsilon
)\subseteq U_{i}\}$. Note that $\eta_{i}$ is continuous and $\eta_{i}(z)>0$ if
and only if $z\in U_{i}$. Thus, $\eta:\partial X\to[0,\infty)$ defined by
$\eta(z)=\text{max}\{\eta_{i}(z)\}_{i=1}^{k}$ is continuous and strictly
positive. Let $\delta^{\prime}$ be the minimum value of $\eta$ and set
$\delta=\frac{\delta^{\prime}}{2}$.
\end{proof}

%We will see in Section 3 that the Baumslag-Solitar groups $BS(m,n)$ for $m\neq n$ act properly and compactly on $T\times \mathbb{R}$ but not isometrically.

\subsection{$\mathcal{Z}$-Structures\label{Subsection: Z-structures}}

Boundaries of CAT(0) groups have proven to be useful objects that can help us
gain more information about the groups themselves. This led Bestvina to
generalize the notion of group boundaries by defining \textquotedblleft%
$\mathcal{Z}$-boundaries\textquotedblright\ for groups, a topic that we
explore now. For more on $\mathcal{Z}$-structures, see \cite{Bes96} and
\cite{GuMo18}.

\begin{definition}
A closed subset $A$ of a space $X$, is a \emph{$\mathcal{Z}$-set} if there
exists a homotopy $H:X\times\lbrack0,1]\rightarrow X$ such that $H_{0}%
=\operatorname*{id}_{X}$ and $H_{t}(X)\subset X-A$ for every $t>0$.
\end{definition}

\begin{example}
The prototypical $\mathcal{Z}$-set is the boundary of a manifold, or any
closed subset of that boundary.
\end{example}

A \emph{\boldmath$\mathcal{Z}$-compactification} of a space $X$ is a
compactification $\overline{X}$ such that $\overline{X}-X$ is a $\mathcal{Z}%
$-set in $\overline{X}$.

\begin{example}
The addition of the visual boundary to a proper CAT(0) space $X$ gives a
$\mathcal{Z}$-compactification $\overline{X}$ of $X$. A simple way to see the
visual boundary as a $\mathcal{Z}$-set in $\overline{X}$ is to imagine the
homotopy that ``reels" points of the boundary in along the geodesic rays.
\end{example}

\begin{definition}
\label{Definition: Z-structure}A\emph{ }$\mathcal{Z}$\emph{-structure }on a
group $G$ is a pair of spaces $(\overline{X},Z)$ satisfying the following four conditions:

\begin{enumerate}
\item $\overline{X}$ is a compact AR,

\item $Z$ is a $\mathcal{Z}$-set in $\overline{X}$,

\item $X=\overline{X}-Z$ is a proper metric space on which $G$ acts
geometrically, and

\item $\overline{X}$ satisfies the following \emph{nullity condition} with
respect to the $G$-action on $X$: for every compact $C\subseteq X$ and any
open cover $\mathcal{U}$ of $\overline{X}$, all but finitely many
$G$-translates of $C$ lie in an element of $\mathcal{U}$.\smallskip
\end{enumerate}

\noindent When this definition is satisfied, $Z$ is called a $\mathcal{Z}%
$\emph{-boundary }for $G$. If only conditions (1)-(3) are satisfied, the
result is called a \emph{weak} $\mathcal{Z}$\emph{-structure}. If, in addition
to (1)-(4)above, the $G$-action on $X$ extends to $\overline{X}$, the result
is called an $\mathcal{EZ}$\emph{-structure} (equivariant) $\mathcal{Z}%
$-structure. \medskip
\end{definition}

\begin{example}
The following are the most common examples of $(\mathcal{E})\mathcal{Z}$-structures:

\begin{enumerate}
\item If $G$ acts geometrically on a proper CAT(0) space $X$, then
$\overline{X}=X\cup\partial_{\infty}X$, with the cone topology, gives an
$\mathcal{EZ}$-structure for $G$.

\item In \cite{BeMe91} it is shown that if $G$ is a hyperbolic group,
$P_{\rho}(G)$ is an appropriately chosen Rips complex, and $\partial G$ is the
Gromov boundary, then $\overline{P}_{\rho}(G)=P_{\rho}(G)\cup\partial G$
(appropriately topologized) gives an $\mathcal{EZ}$-structure for $G$.

\item Osajda and Przytycki \cite{OsPr09} have shown that systolic groups admit
$\mathcal{EZ}$-structures.
\end{enumerate}
\end{example}

Other classes of groups that admit $\mathcal{Z}$-structures have been
addressed by Dahmani \cite{Dah03} (relatively hyperbolic groups), Martin
\cite{Mar14} (complexes of groups), Osajda and Przytycki \cite{OsPr09}
(systolic groups), Tirel \cite{Tir11} (free and direct products), and Pietsch
\cite{Pie18} (semidirect products with $\mathbb{Z}$ and 3-manifold groups).

Most of the Baumslag-Solitar groups $BS(m,n)$ and generalized Baumslag-Solitar
groups do not belong to any of the categories listed above and thus Theorem
\ref{Theorem: Boundary BS(m,n)} adds an interesting new set of examples to
this list.

A few comments are in order regarding the definition of $\mathcal{Z}%
$-structure. First, Bestvina's original definition did not explicitly require
actions by isometries, but only by covering transformations. As we point out
at the end of Section \ref{Subsection: the adjusted action}, there is no loss
of generality in requiring actions by isometries. Bestvina also required
$\overline{X}$ to be finite-dimensional and the action to be free.
Dranishnikov relaxed both of these conditions in \cite{Dra06}, and
\cite{GuMo18} shows that nothing is lost in doing so.

We close this section with a few observations about $\mathcal{Z}$-structures.
The first makes the nullity condition more intuitive; the second is useful for
verifying the nullity condition; and the third can (and will) be used to
obtain $\mathcal{Z}$-structures for a broad class of groups without checking
each group individually.\medskip

Every $\mathcal{Z}$-compactification $\overline{X}$ of a proper metric space
$\left(  X,d\right)  $ is metrizable (see \cite{GuMo18}), but in general,
there is no canonical choice of metric for $\overline{X}$; moreover whichever
metric $\overline{d}$ one chooses will be quite different from $d$.
Nevertheless, any such choice can be used to give the following intuitive
meaning to the nullity condition. The proof is straight-forward general topology.

\begin{lemma}
Let $(\overline{X},Z)$ be a weak $\mathcal{Z}$-structure as described in
Definition \ref{Definition: Z-structure}, and let $\overline{d}$ be a metric
for $\overline{X}$. Then $(\overline{X},Z)$ satisfies the nullity condition
(and hence is a $\mathcal{Z}$-structure) if and only if
\begin{gather*}
\text{(\dag) for any compact set }C\subseteq X\text{ and }\epsilon>0\text{,all
but finitely many }G\text{-translates}\\
\text{ of }C\text{\textit{ have }}\overline{d}\text{-diameter less than
}\epsilon\text{.}%
\end{gather*}

\end{lemma}

The next lemma allows us to verify the nullity condition without checking
every compact subset $C$ of $X$.

\begin{lemma}
\label{Lemma: Nullity Condition Check}Let $X$ be a proper metric space
admitting a proper cocompact action by $G$ and let $\left(  \overline
{X},\overline{d}\right)  $ be a $\mathcal{Z}$-compactification of $X$ If $C$
is a compact subset of $X$ with the property that $GC=X$ and the nullity
condition is satisfied for $C$, then the nullity condition is satisfied for
all compact subsets of $X$.
\end{lemma}

\begin{proof}
Choose $\epsilon>0$ and let $K\subseteq X$ be an arbitrary compact set. By
properness and the hypothesis, there are finitely many translates of $C$ that
cover $K$, that is $K\subseteq g_{1}C\cup g_{2}C\cup...\cup g_{n}C$ for
$g_{i}\in G$. Since $C$ satisfies the nullity condition, all but finitely many
$G$ translates of $C$ have $\overline{d}$-diameter less than $\frac{\epsilon
}{n}$. If we consider any translate $gK$, then $gK\subseteq gg_{1}C\cup
gg_{2}C\cup...\cup gg_{n}C$. Only finitely many $gg_{i}C$ for $g\in G$ have
diameter greater than $\frac{\epsilon}{n}$ and thus only finitely many $gK$
have diameter greater than $n\frac{\epsilon}{n}=\epsilon$.
\end{proof}

The following useful fact is often referred to as the \textquotedblleft
boundary swapping trick\textquotedblright.

\begin{proposition}
\label{Boundary Swapping Theorem} \cite{Bes96, GuMo18} Suppose $G$ and $H$ are
quasi-isometric groups that act geometrically on proper metric ARs $X$ and
$Y$, respectively, and $Y$ can be compactified to a $\mathcal{Z}$-structure
$\left(  \overline{Y},Z\right)  $ for $H$, then $X$ can be compactified by
addition of the same boundary to obtain a $\mathcal{Z}$-structure $\left(
\overline{X},Z\right)  $ for $G$.
\end{proposition}

\section{$\mathcal{Z}$-structures on generalized Baumslag-Solitar Groups}

A \emph{Baumslag-Solitar group} $BS(m,n)$ is a two generator, one relator
group admitting a presentation of the form%
\[
BS(m,n)=\left\langle s,t\,|\,ts^{m}t^{-1}=s^{n}\right\rangle .
\]
Without loss of generality, we may assume that $0<\left\vert m\right\vert \leq
n$. These groups are HNN extensions of $\mathbb{Z}$ with infinite cyclic
associated subgroups, and the standard presentation 2-complex $K_{m,n}$ is a
$K\left(  \pi,1\right)  $ space. If we begin with the canonical graph of
groups representation of $BS\left(  m,n\right)  $ with one vertex and one edge,
the corresponding Bass-Serre tree is the directed tree $T\left(  \left\vert
m\right\vert ,n\right)  $ with $\left\vert m\right\vert $ incoming and $n$
outgoing edges at each vertex, and the universal cover of $K_{m,n}$ is
homeomorphic to $%
%TCIMACRO{\U{211d} }%
%BeginExpansion
\mathbb{R}
%EndExpansion
\times T\left(  \left\vert m\right\vert ,n\right)  $. Gersten \cite{Ger92} has
shown that, provided $\left\vert m\right\vert \neq n$, the Dehn function of
$BS\left(  m,n\right)  $ is not bounded by a polynomial. By contrast, Dehn
functions of hyperbolic and CAT(0) groups are bounded by linear and quadratic
functions, respectively. So most Baumslag-Solitar groups are neither
hyperbolic nor CAT(0). As such, this collection of groups contains some of the
simplest candidates for $\mathcal{Z}$-structures not covered by the motivating examples.

%These groups act properly and cocompactly on $T\times\mathbb{R}$ (where $T$ is a tree), a CAT(0) space, but not isometrically. In fact, rectangles get arbitrarily large in the vertical direction of a sheet, also violating the nullity condition. The main work in proving this theorem is to address these two issues.

\subsection{Generalized Baumslag-Solitar
groups\label{Subsection: Generalized Baumslag-Solitar groups}}

A \emph{generalized Baumslag-Solitar group} is the fundamental group $G$ of a
finite graph of groups with all vertex and edge groups $\mathbb{Z}$. In
\cite{Why01}, Whyte classified generalized Baumslag-Solitar groups, up to
quasi-isometry.
%Many of these groups are neither CAT(0) nor hyperbolic (Citation????). Baumslag-Solitar groups $BS(m,n)$ are contained in this class as they are HNN extensions of $\mathbb{Z}$.

\begin{theorem}
\label{Theorem: Whyte Classification}\cite{Why01} If $\Gamma$ is a graph of
$\mathbb{Z}$s and $G=\pi_{1}\Gamma$, then exactly one of the following is true:

\begin{enumerate}
\item $G$ contains a subgroup of finite index of the form $\mathbb{Z\times
F}_{n}$

\item $G=BS(1,n)$ for some $n>1$

\item $G$ is quasi-isometric to $BS(2,3)$.
\end{enumerate}
\end{theorem}

As with the ordinary Baumslag-Solitar groups, each generalized
Baumslag-Solitar group $G$ acts properly and cocompactly on $%
%TCIMACRO{\U{211d} }%
%BeginExpansion
\mathbb{R}
%EndExpansion
\mathbb{\times}T$ where $T$ is the Bass-Serre tree of its graph of groups
representation. If $G$ is of the first type mentioned in Theorem
\ref{Theorem: Whyte Classification}, it is quasi-isometric to the CAT(0) group
$%
%TCIMACRO{\U{2124} }%
%BeginExpansion
\mathbb{Z}
%EndExpansion
\times F_{n}$; so by the boundary swapping trick, (Proposition
\ref{Boundary Swapping Theorem}), $G$ admits a $\mathcal{Z}$-structure. By
another application of Theorem \ref{Theorem: Whyte Classification} and the
boundary swapping trick, we can then obtain $\mathcal{Z}$-structures for all
generalized Baumslag-Solitar groups, provided we can obtain them for ordinary
Baumslag-Solitar groups. That is where we turn our attention now.

\subsection{A \textquotedblleft standard\textquotedblright\ action of
$BS\left(  m,n\right)  $ on $%
%TCIMACRO{\U{211d} }%
%BeginExpansion
\mathbb{R}
%EndExpansion
\times T\left(  \left\vert m\right\vert ,n\right)
\label{Subsection: standard action}$}

As noted above, $BS\left(  m,n\right)  $ acts properly, freely, and
cocompactly on $%
%TCIMACRO{\U{211d} }%
%BeginExpansion
\mathbb{R}
%EndExpansion
\times T\left(  \left\vert m\right\vert ,n\right)  $. In Example
\ref{Example: CAT(0) boundaries}, we observed that this space admits a
$\mathcal{Z}$-compactification by addition of the suspension of $\partial
_{\infty}T(\left\vert m\right\vert ,n)$. That is accomplished by giving $%
%TCIMACRO{\U{211d} }%
%BeginExpansion
\mathbb{R}
%EndExpansion
\times T\left(  \left\vert m\right\vert ,n\right)  $ its natural CAT(0) metric
and adding the visual boundary. This gives us a weak $\mathcal{Z}$-structure
for $BS\left(  m,n\right)  $, but since the action of $BS\left(  m,n\right)  $
on this CAT(0) space is not by isometries, the nullity condition does not
follow. In fact, if we subdivide $%
%TCIMACRO{\U{211d} }%
%BeginExpansion
\mathbb{R}
%EndExpansion
\times T\left(  \left\vert m\right\vert ,n\right)  $ into rectangular
principal domains for $BS\left(  m,n\right)  $ in the traditional manner (see
Figure 1) and if $\left\vert m\right\vert \neq n$, these rectangles grow
exponentially as they are translated along the positive $t$-axis. More
importantly (for our purposes), translates of the fundamental domain remain
large in the compactification (details to follow). Arranging the nullity
condition will require significantly more work.

Although this \textquotedblleft standard\textquotedblright\ action of
$BS\left(  m,n\right)  $ on $%
%TCIMACRO{\U{211d} }%
%BeginExpansion
\mathbb{R}
%EndExpansion
\times T\left(  \left\vert m\right\vert ,n\right)  $ with its CAT(0) metric
and corresponding visual boundary does not give the desired $\mathcal{(E)Z}%
$-structure, the picture it provides is useful; therefore we supply some
additional details.

For the moment it is convenient to assume that $m>0$. Choose a preferred
vertex $v_{0}$ of $T\left(  \left\vert m\right\vert ,n\right)  $ and place the
Cayley graph $\Gamma$ of $BS\left(  m,n\right)  $ in $%
%TCIMACRO{\U{211d} }%
%BeginExpansion
\mathbb{R}
%EndExpansion
\times T(\left\vert m\right\vert ,n)$ so that $\mathbf{v}_{0}=\left(
0,v_{0}\right)  $ corresponds to $1\in B\left(  m,n\right)  $, and the
positively oriented edge-ray $\tau^{+}\subseteq\Gamma$ whose edges are each
labeled by an outward pointing $t$ and the negatively oriented edge-ray
$\tau^{-}$whose edges are each labeled by an inward pointing $t$ both lie in
$\left\{  0\right\}  \times T(\left\vert m\right\vert ,n)$. In other words,
the line $\tau\equiv\tau^{-}\cup\tau^{+}\subseteq\Gamma$, corresponding to the
subgroup $\left\langle t\right\rangle $, is a subset of $\left\{  0\right\}
\times T(\left\vert m\right\vert ,n)$. Subdivide $%
%TCIMACRO{\U{211d} }%
%BeginExpansion
\mathbb{R}
%EndExpansion
\times\left\{  v_{0}\right\}  $ into edges of length $1/n$, each oriented in
the positive $%
%TCIMACRO{\U{211d} }%
%BeginExpansion
\mathbb{R}
%EndExpansion
$-direction and labeled by the generator $s$. Thus we have identified this
line with the subgroup $\left\langle s\right\rangle $. Let $R_{0}\subseteq%
%TCIMACRO{\U{211d} }%
%BeginExpansion
\mathbb{R}
%EndExpansion
\times\tau$ be the $1\times1$ rectangle with lower left-hand vertex at $1$ and
boundary labeled by the defining relator of $BS\left(  m,n\right)  $. Tile the
plane $%
%TCIMACRO{\U{211d} }%
%BeginExpansion
\mathbb{R}
%EndExpansion
\times\tau$ with rectangular fundamental domains, each of whose boundaries is
labeled by the relator as shown in Figure 1, keeping in mind that this plane
represents only a small portion of the Cayley complex.
\begin{figure}[h]
	\includegraphics[trim=500mm 90mm 0mm 30mm, scale=0.3]{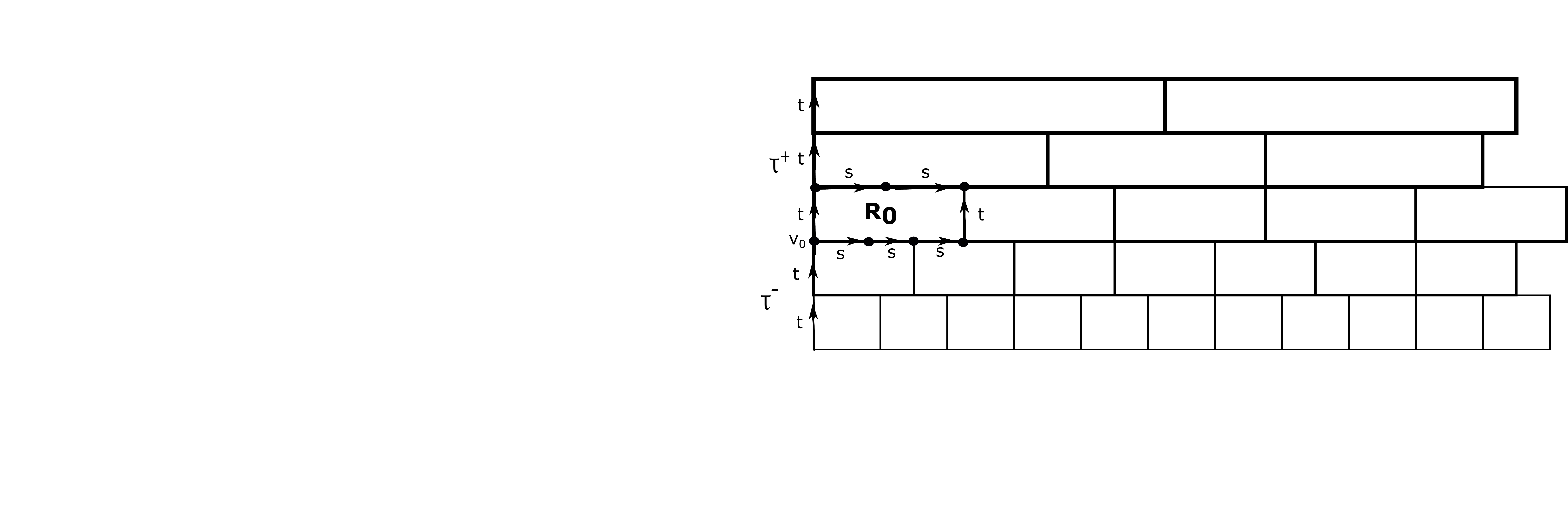}
	\caption{Tiling of BS(2,3)}
\end{figure} 

For each edge-ray $\rho\subseteq T(\left\vert m\right\vert ,n)$ emanating from
$v_{0}$, we refer to the half-plane $%
%TCIMACRO{\U{211d} }%
%BeginExpansion
\mathbb{R}
%EndExpansion
\times\rho$ as a \emph{sheet} of $%
%TCIMACRO{\U{211d} }%
%BeginExpansion
\mathbb{R}
%EndExpansion
\times T\left(  \left\vert m\right\vert ,n\right)  $. If all edges on $\rho$
are positively oriented, call $%
%TCIMACRO{\U{211d} }%
%BeginExpansion
\mathbb{R}
%EndExpansion
\times\rho$ a \emph{positive sheet}; if all edges are negatively oriented,
call $%
%TCIMACRO{\U{211d} }%
%BeginExpansion
\mathbb{R}
%EndExpansion
\times\rho$ a \emph{negative sheet}; and if $\rho$ contains both orientations,
call $%
%TCIMACRO{\U{211d} }%
%BeginExpansion
\mathbb{R}
%EndExpansion
\times\rho$ a \emph{mixed sheet}. Call $%
%TCIMACRO{\U{211d} }%
%BeginExpansion
\mathbb{R}
%EndExpansion
\times\tau^{+}$ the \emph{preferred positive sheet and }$%
%TCIMACRO{\U{211d} }%
%BeginExpansion
\mathbb{R}
%EndExpansion
\times\tau^{-}$ the \emph{preferred negative sheet}. (Note: Although the
oriented tree $\left\{  0\right\}  \times T(\left\vert m\right\vert ,n)$ plays
a useful role, most of its edges are not contained in $\Gamma$.)

Notice that each sheet is a convex subset of $%
%TCIMACRO{\U{211d} }%
%BeginExpansion
\mathbb{R}
%EndExpansion
\times T\left(  \left\vert m\right\vert ,n\right)  $ isometric to a Euclidean
half-plane. Up to horizontal translation, all positive sheets inherit a tiling
identical to that of $%
%TCIMACRO{\U{211d} }%
%BeginExpansion
\mathbb{R}
%EndExpansion
\times\tau^{+}$ and all negative sheets inherit a tiling identical (up to
translation) to $%
%TCIMACRO{\U{211d} }%
%BeginExpansion
\mathbb{R}
%EndExpansion
\times\tau^{-}$. So, in positive sheets the widths of the fundamental domains
increase (exponentially) as one gets further from $v_{0}$ in the $T(\left\vert
m\right\vert ,n)$-direction, while in the negative sheets the widths decrease.
In mixed sheets, widths do not change in a monotone manner---sometimes they
increase and sometimes they decrease; but the resulting tiling is always finer
than that of an appropriately placed positive sheet. In other words, the tiles
in a generic sheet always fit inside those of a correspondingly subdivided
positive sheet. Finally, note also that for $m<0$, the tiling is the same, but
with the $s$ edges at odd integer heights oriented in the negative $%
%TCIMACRO{\U{211d} }%
%BeginExpansion
\mathbb{R}
%EndExpansion
$-direction.

\subsection{An adjusted action of $BS\left(  m,n\right)  $ on $%
%TCIMACRO{\U{211d} }%
%BeginExpansion
\mathbb{R}
%EndExpansion
\times T\left(  \left\vert m\right\vert ,n\right)
\label{Subsection: the adjusted action}$}

Under the above setup, the nullity condition fails badly. For example,
translates of $R_0$ by powers of $t$ limit out on the entire quarter circle
bounding the right-hand quadrant of $%
%TCIMACRO{\U{211d} }%
%BeginExpansion
\mathbb{R}
%EndExpansion
\mathbb{\times\tau}^{+}$ in the visual compactification of the CAT(0) space $%
%TCIMACRO{\U{211d} }%
%BeginExpansion
\mathbb{R}
%EndExpansion
\times T\left(  \left\vert m\right\vert ,n\right) $. Instead of changing the
space or its compactification, we will remedy this problem by changing the
action. Some of the resulting calculations are lengthy, but the idea is
simple. Define $f:%
%TCIMACRO{\U{211d} }%
%BeginExpansion
\mathbb{R}
%EndExpansion
\times T\left(  \left\vert m\right\vert ,n\right)  \rightarrow%
%TCIMACRO{\U{211d} }%
%BeginExpansion
\mathbb{R}
%EndExpansion
\times T\left(  \left\vert m\right\vert ,n\right)  $ by
\[
f(x,y)=(\operatorname{sgn}\left(  x\right)  \log(\log(|x|+e)),y)
\]
Our new action is via conjugation by this homeomorphism. More specifically,
for each $g\in BS(m,n)$, viewed as a self-homeomorphism of $%
%TCIMACRO{\U{211d} }%
%BeginExpansion
\mathbb{R}
%EndExpansion
\times T\left(  \left\vert m\right\vert ,n\right)  $ under the original
$BS(m,n)$ action, define $\overline{g}:%
%TCIMACRO{\U{211d} }%
%BeginExpansion
\mathbb{R}
%EndExpansion
\times T\left(  \left\vert m\right\vert ,n\right)  \rightarrow%
%TCIMACRO{\U{211d} }%
%BeginExpansion
\mathbb{R}
%EndExpansion
\times T\left(  \left\vert m\right\vert ,n\right)  $ by $\overline{g}=f\circ
g\circ f^{-1}$. Here $f^{-1}:%
%TCIMACRO{\U{211d} }%
%BeginExpansion
\mathbb{R}
%EndExpansion
\times T\left(  \left\vert m\right\vert ,n\right)  \rightarrow%
%TCIMACRO{\U{211d} }%
%BeginExpansion
\mathbb{R}
%EndExpansion
\times T\left(  \left\vert m\right\vert ,n\right)  $ can be specified by:%
\[
f^{-1}(x,y)=\left(  \operatorname{sgn}\left(  x\right)  (\exp\left(
\exp\left(  \left\vert x\right\vert \right)  \right)  -e\right)  ,y)
\]
For simplicity, we refer to this as the $\overline{BS(m,n)}$-action on $%
%TCIMACRO{\U{211d} }%
%BeginExpansion
\mathbb{R}
%EndExpansion
\times T\left(  \left\vert m\right\vert ,n\right)  $. Our goal then is to show
that \emph{with this action}, the visual compactification of $%
%TCIMACRO{\U{211d} }%
%BeginExpansion
\mathbb{R}
%EndExpansion
\times T\left(  \left\vert m\right\vert ,n\right)  $ satisfies the definition
of $\mathcal{Z}$-structure. After that task is completed, we will show that
this action also extends to the visual boundary, thereby completing the proof
of Theorem \ref{Theorem: Boundary BS(m,n)}.

Before proceeding with the calculations, note that the $\overline{BS(m,n)}%
$-action on the CAT(0) space $%
%TCIMACRO{\U{211d} }%
%BeginExpansion
\mathbb{R}
%EndExpansion
\times T\left(  \left\vert m\right\vert ,n\right)  $ is still not by
isometries---as noted earlier, that would be impossible since $BS\left(
m,n\right)  $ is not CAT(0) when $\left\vert m\right\vert \neq n$. To obtain
the isometry requirement implicit in Definition \ref{Definition: Z-structure}
we can apply the following proposition. It reveals that the isometry
requirement is mostly just a technicality.

\begin{proposition}
\label{Propostion:Geometric Action}\cite{AMN11} Suppose $G$ acts properly and
cocompactly on a locally compact space $X$. Then there is a topologically
equivalent proper metric for $X$ under which the action is by
isometries.\bigskip
\end{proposition}

\subsection{Nullity condition for the $\overline{BS\left(  m,n\right)  }%
$-action on $%
%TCIMACRO{\U{211d} }%
%BeginExpansion
\mathbb{R}
%EndExpansion
\times T\left(  \left\vert m\right\vert ,n\right)  $}

Recall the $1\times1$ rectangle $R_{0}\subseteq%
%TCIMACRO{\U{211d} }%
%BeginExpansion
\mathbb{R}
%EndExpansion
\times\tau$ defined earlier. Under the standard action of $BS\left(
m,n\right)  $ on $%
%TCIMACRO{\U{211d} }%
%BeginExpansion
\mathbb{R}
%EndExpansion
\times T\left(  \left\vert m\right\vert ,n\right)  $ acts as our preferred
fundamental domain. Translates of $R_{0}$ by elements of $BS\left(  m,n\right)
$ produce a \textquotedblleft tiling\textquotedblright\ of $%
%TCIMACRO{\U{211d} }%
%BeginExpansion
\mathbb{R}
%EndExpansion
\times T\left(  \left\vert m\right\vert ,n\right)  $, part of which is
pictured in Figure 1. The most notable trait of this tiling is that, while
the heights of all rectangles in the tiling are $1$ (measured along the
$T\left(  \left\vert m\right\vert ,n\right)  $-coordinate), the widths of
rectangles in the positive sheets grow exponentially with the $T\left(
\left\vert m\right\vert ,n\right)  $-coordinate whenever $\left\vert
m\right\vert \neq n$. For example, a generic tile in a positive sheet with
lower edge at height $b$ will have width $\left(  \frac{m}{\left\vert
n\right\vert }\right)  ^{b}$. Widths of tiles in generic sheets are bounded
above by this number. Under the $\overline{BS(m,n)}$-action on $%
%TCIMACRO{\U{211d} }%
%BeginExpansion
\mathbb{R}
%EndExpansion
\times T\left(  \left\vert m\right\vert ,n\right)  $, the role of $R_{0}$ is
played by the compressed rectangle $\overline{R}_{0}=f\left(  R_{0}\right)  $,
and every $\overline{BS(m,n)}$-tile has its width compressed by the $\log\log$
function. Most importantly, for \ the sake of calculations, a generic
$\overline{BS(2,3)}$-tile in the preferred positive sheet will have the
coordinates shown in Figure 2. For a generic $\overline{BS(m,n)}$-tile, simply replace $2$ and $3$ by $m$ and $n$, respectively. 
\begin{figure}[h]
\includegraphics[scale=0.4]{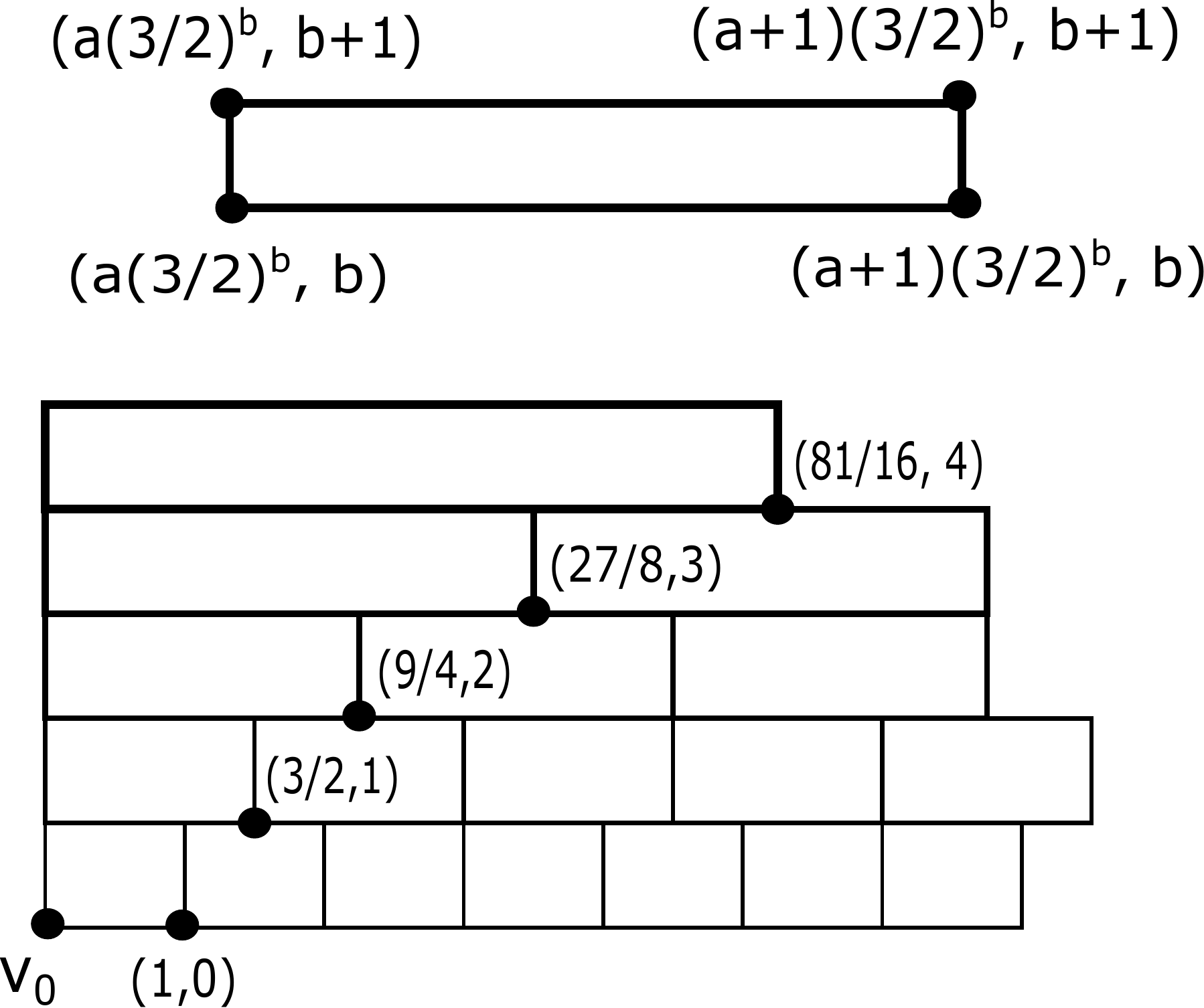}
\caption{Coordinates of $BS(2,3)$}
\end{figure} 

For a CAT(0) space $X$, the reason $\partial_{\infty}X$ is called the
\textquotedblleft visual boundary\textquotedblright\ is because, in a flat
geometry, the size of a set $A\subseteq X$ viewed within $\overline{X}$ is
related to the angle of vision it subtends for a viewer stationed at a
fixed origin. For that reason (with more precision to be provided shortly),
the following lemma and its corollary are key. To keep calculations as simple
as possible, we begin by analyzing the preferred positive sheet of $%
%TCIMACRO{\U{211d} }%
%BeginExpansion
\mathbb{R}
%EndExpansion
\times T\left(  \left\vert m\right\vert ,n\right)  $.

\begin{lemma}
\label{Theorem: Scaling Function Property}For each $\epsilon>0$, there exists
$M_{\epsilon}>0$ such that if $\overline{g}\overline{R}_{0}$ is a
$\overline{BS(m,n)}$-tile lying in the preferred positive sheet $%
%TCIMACRO{\U{211d} }%
%BeginExpansion
\mathbb{R}
%EndExpansion
\mathbb{\times\tau}^{+}$ of $%
%TCIMACRO{\U{211d} }%
%BeginExpansion
\mathbb{R}
%EndExpansion
\times T\left(  \left\vert m\right\vert ,n\right)  $ and outside the closed
$M_{\epsilon}$-ball of $%
%TCIMACRO{\U{211d} }%
%BeginExpansion
\mathbb{R}
%EndExpansion
\times T\left(  \left\vert m\right\vert ,n\right)  $ centered at
$\mathbf{v}_{0}$, and if $\mathbf{w}_{1},\mathbf{w}_{2}\in\overline
{g}\overline{R}_{0}$, then the angle between segments $\overline
{\mathbf{v}_{0}\mathbf{w}_{1}}$ and $\overline{\mathbf{v}_{0}\mathbf{w}_{2}}$
is less than $\epsilon$.
\end{lemma}

\begin{proof}
First note that $%
%TCIMACRO{\U{211d} }%
%BeginExpansion
\mathbb{R}
%EndExpansion
\mathbb{\times\tau}^{+}$ is a Euclidean half-plane, so angle refers to
standard angle measure. Similarly, since $%
%TCIMACRO{\U{211d} }%
%BeginExpansion
\mathbb{R}
%EndExpansion
\mathbb{\times\tau}^{+}$ is a convex subset of $%
%TCIMACRO{\U{211d} }%
%BeginExpansion
\mathbb{R}
%EndExpansion
\times T\left(  \left\vert m\right\vert ,n\right)  $ with $\mathbf{v}_{0}$
corresponding to the origin, a closed $M_{\epsilon}$-ball of $%
%TCIMACRO{\U{211d} }%
%BeginExpansion
\mathbb{R}
%EndExpansion
\times T\left(  \left\vert m\right\vert ,n\right)  $ intersects $%
%TCIMACRO{\U{211d} }%
%BeginExpansion
\mathbb{R}
%EndExpansion
\mathbb{\times\tau}^{+}$ precisely in the closed half-disk of the same radius.
As such, Figure 3 accurately captures the situation.

Since our tiling is symmetric about the vertical axis, we may assume that
$\overline{g}\overline{R}_{0}$ lies in the right-hand quadrant and has
vertices with Euclidean coordinates:

\begin{itemize}
\item $(\log(\log(a(\frac{n}{\left\vert m\right\vert })^{b}+e)),b)$

\item $(\log(\log((a+1)(\frac{n}{\left\vert m\right\vert })^{b}+e)),b)$

\item $(\log(\log(a(\frac{n}{\left\vert m\right\vert })^{b}+e)),b+1)$

\item $(\log(\log((a+1)(\frac{n}{\left\vert m\right\vert })^{b}+e)),b+1)$
\end{itemize}

\noindent where all numbers in the formulae, except possibly $m$, are non-negative.

For simplicity of notation, let $p=\log(\log(a(\frac{n}{|m|})^{b}+e))$ and
$q=\log(\log((a+1)(\frac{n}{|m|})^{b}+e))$. Note that the angle between
$\overline{\mathbf{v}_{0}\mathbf{w}_{1}}$ and $\overline{\mathbf{v}%
_{0}\mathbf{w}_{2}}$ is no larger than the angle between segments
$\overline{\mathbf{v}_{0},(p,b+1)}$ and $\overline{\mathbf{v}_{0},(q,b)}$.

\begin{figure} [h]
	\includegraphics[scale=0.5]{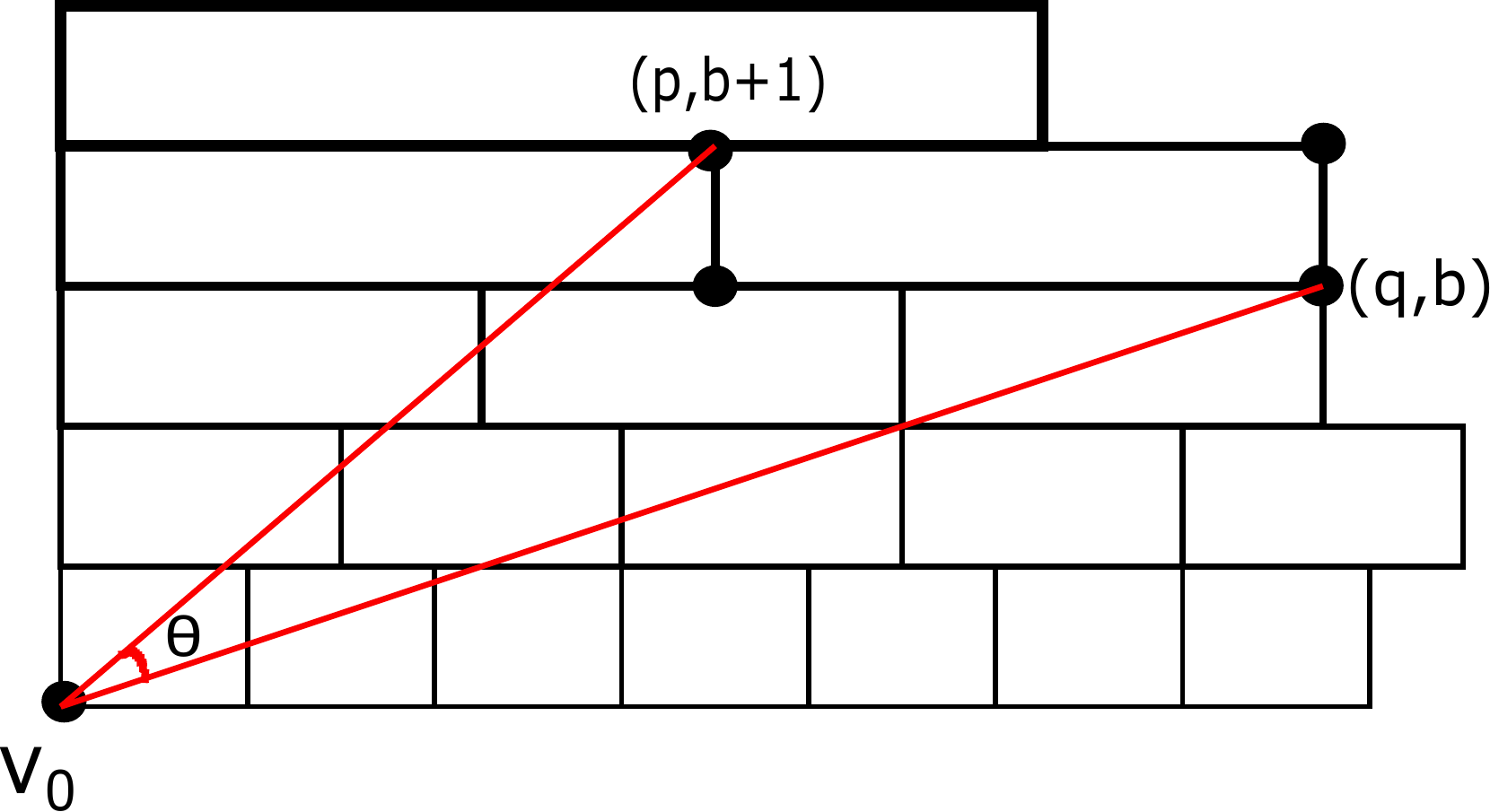}
	\caption{Angle Measurement in Preferred Sheet of $BS(2,3)$}
\end{figure}

Representing that angle by $\theta$, we have the formula.%

\[
\theta=\tan^{-1}\left(  \frac{b+1}{p}\right)  -\tan^{-1}\left(  \frac{b}%
{q}\right)
\]
and by application of a few inverse tangent identities:%

\begin{align*}
\theta &  =\tan^{-1}\left(  \frac{b+1}{p}\right)  +\tan^{-1}\left(  \frac
{-b}{q}\right) \\
&  =\tan^{-1}\left(  \frac{\frac{b+1}{p}+\frac{-b}{q}}{1-\frac{-b(b+1)}{pq}%
}\right)  .
\end{align*}
By algebraic manipulation we then obtain:
\begin{align*}
\theta &  =\tan^{-1}\left(  \frac{(b+1)q-bp}{pq+b^{2}+b}\right) \\
&  =\tan^{-1}\left(  \frac{b(q-p)+q}{pq+b^{2}+b}\right) \\
&  =\tan^{-1}\left(  \frac{b(q-p)}{pq+b^{2}+b}+\frac{q}{pq+b^{2}+b}\right)
\end{align*}

Next we analyze this formula when $a$ and/or $b$ get large. Recall that $p$
and $q$ are both defined in terms of $a$ and $b$. In particular as one of $a$
and $b$ or both get large, $p$ and $q$ get large. Thus, the second term in the
above sum clearly gets small as $a$ or $b$ get large. So, to deduce that
$\theta$ approaches $0$ as$\sqrt{a^{2}+b^{2}}$ gets large, we need only check
that the first term in that sum goes to zero.

We direct our attention to proving that the term
\begin{equation}
\frac{b(q-p)}{pq+b^{2}+b} \tag{\#}\label{Key quantity}%
\end{equation}
approaches $0$ as $\sqrt{a^{2}+b^{2}}$ tends to infinity.

Recall that%

\begin{align*}
b(q-p)  &  =b(\log(\log((a+1)\left(\frac{n}{m}\right)^{b}+e))-\log(\log(a\left(\frac{n}%
{m}\right)^{b}+e)))\\
&  =b\log\left(  \frac{\log((a+1)\left(  \frac{n}{m}\right)  ^{b}+e)}%
{\log(a\left(  \frac{n}{m}\right)  ^{b}+e)}\right)
\end{align*}

We split our analysis into four cases, applying L'H\^{o}pital's Rule when
appropriate. For simplicity of notation, we assume $m>0$; if not, replace $m$ by $|m|$ in 
the calculations below: \medskip

\noindent\underline{Case 1:} $b$\emph{ is bounded and }$a$\emph{ gets
large.}\medskip%

\[
\lim_{a\rightarrow\infty}\frac{\log\left(  (a+1)\left(  \frac{n}{m}\right)
^{b}+e\right)  }{\log\left(  a\left(  \frac{n}{m}\right)  ^{b}+e\right)
}=\lim_{a\rightarrow\infty}\frac{\left(  \frac{n}{m}\right)  ^{b}%
}{(a+1)\left(  \frac{n}{m}\right)  ^{b}+e}\cdot\frac{a\left(  \frac{n}%
{m}\right)  ^{b}+e}{\left(  \frac{n}{m}\right)  ^{b}}=1
\]
Thus,
\[
b\log\left(  \frac{\log\left(  (a+1)\left(  \frac{n}{m}\right)  ^{b}+e\right)
}{\log\left(  a\left(  \frac{n}{m}\right)  ^{b}+e\right)  }\right)
\rightarrow0
\]
\medskip

\noindent\underline{Case 2:} $a=0$ \emph{and }$b$\emph{ gets large.}\medskip\
\[
b\log\left(  \log\left(  \left(  \frac{n}{m}\right)  ^{b}+e\right)  \right)
\]%
\[
\sim b\log\left(  b\right)
\]
and since there is a $b^{2}$ term in the denominator, (\ref{Key quantity})
approaches $0$, as desired.\medskip

\noindent\underline{Case 3:} $a$\emph{ is bounded and }$b$\emph{ gets
large.}\medskip

\noindent Then\
\begin{align*}
\lim_{b\rightarrow\infty}\frac{\log\left(  (a+1)\left(  \frac{n}{m}\right)
^{b}+e\right)  }{\log\left(  a\left(  \frac{n}{m}\right)  ^{b}+e\right)  } &
=\lim_{b\rightarrow\infty}\frac{(a+1)\left(  \frac{n}{m}\right)  ^{b}%
\log\left(  \frac{n}{m}\right)  }{(a+1)\left(  \frac{n}{m}\right)  ^{b}%
+e}\cdot\frac{a\left(  \frac{n}{m}\right)  ^{b}+e}{a\left(  \frac{n}%
{m}\right)  ^{b}\log\left(  \frac{n}{m}\right)  }\\
&  =\lim_{b\rightarrow\infty}\frac{(a+1)\left(  \frac{n}{m}\right)  ^{b}%
\log\left(  \frac{n}{m}\right)  }{a\left(  \frac{n}{m}\right)  ^{b}\log\left(
\frac{n}{m}\right)  }\cdot\frac{a\left(  \frac{n}{m}\right)  ^{b}%
+e}{(a+1)\left(  \frac{n}{m}\right)  ^{b}+e}\\
&  =\frac{a+1}{a}\cdot\frac{a}{a+1}\\
&  =1
\end{align*}
So
\[
\log\left(  \frac{\log\left(  (a+1)\left(  \frac{n}{m}\right)  ^{b}+e\right)
}{\log\left(  a\left(  \frac{n}{m}\right)  ^{b}+e\right)  }\right)
\rightarrow0
\]
and hence,
\[
b\log\left(  \frac{\log\left(  (a+1)\left(  \frac{n}{m}\right)  ^{b}+e\right)
}{\log\left(  a\left(  \frac{n}{m}\right)  ^{b}+e\right)  }\right)
\]
grows slower than $b$. It follows easily that (\ref{Key quantity}) again
approaches $0$.\medskip

\noindent\underline{Case 4:} $a$\emph{ and }$b$\emph{ both get large.}\medskip

First notice that%

\[
\frac{\log\left(  (a+1)\left(  \frac{n}{m}\right)  ^{b}+e\right)  }%
{\log\left(  a\left(  \frac{n}{m}\right)  ^{b}+e\right)  }\leq\frac
{\log\left(  2a\left(  \frac{n}{m}\right)  ^{b}+e\right)  }{\log\left(
a\left(  \frac{n}{m}\right)  ^{b}+e\right)  }%
\]
Set $x=a\left(  \frac{n}{m}\right)  ^{b}$. Then as $a,b\rightarrow\infty$,
$x\rightarrow\infty$, hence:
\[
\lim_{a,b\rightarrow\infty}\frac{\log\left(  2a\left(  \frac{n}{m}\right)
^{b}+e\right)  }{\log\left(  a\left(  \frac{n}{m}\right)  ^{b}+e\right)
}=\lim_{x\rightarrow\infty}\frac{\log\left(  2x+e\right)  }{\log(x+e)}%
=\lim_{x\rightarrow\infty}\frac{2}{2x+e}\cdot\frac{x+e}{1}=1
\]
And thus,
\[
b\log\left(  \frac{\log\left(  (a+1)\left(  \frac{n}{m}\right)  ^{b}+e\right)
}{\log\left(  a\left(  \frac{n}{m}\right)  ^{b}+e\right)  }\right)
\]
grows no faster than $b$. Again, we conclude that (\ref{Key quantity})
approaches $0$.

\end{proof}

Now suppose $\overline{g}\overline{R}_{0}$ is an arbitrary tile of $%
%TCIMACRO{\U{211d} }%
%BeginExpansion
\mathbb{R}
%EndExpansion
\times T\left(  \left\vert m\right\vert ,n\right)  $. We may choose an edge
ray $\rho$ in $T\left(  \left\vert m\right\vert ,n\right)  $ emanating from
$\nu_{0}$ so that $\overline{g}\overline{R}_{0}$ lies in the sheet $%
%TCIMACRO{\U{211d} }%
%BeginExpansion
\mathbb{R}
%EndExpansion
\times\rho$, which inherits the geometry of a Euclidean half-plane with
$\mathbf{v}_{0}$ at the origin. For points $\mathbf{w}_{1},\mathbf{w}_{2}%
\in\overline{g}\overline{R}_{0}$ we can measure the angle between segments
$\overline{\mathbf{v}_{0}\mathbf{w}_{1}}$ and $\overline{\mathbf{v}%
_{0}\mathbf{w}_{2}}$ in this half-plane. That measure does not depend on the
sheet chosen.

\begin{corollary}
\label{Corollary: the angle lemma}For each $\epsilon>0$, there exists
$N_{\epsilon}>0$ such that if $\overline{g}\overline{R}_{0}$ is a
$\overline{BS(m,n)}$-tile of $%
%TCIMACRO{\U{211d} }%
%BeginExpansion
\mathbb{R}
%EndExpansion
\times T\left(  \left\vert m\right\vert ,n\right)  $ lying outside the closed
$N_{\epsilon}$-ball of $%
%TCIMACRO{\U{211d} }%
%BeginExpansion
\mathbb{R}
%EndExpansion
\times T\left(  \left\vert m\right\vert ,n\right)  $ centered at
$\mathbf{v}_{0}$, and if $\mathbf{w}_{1},\mathbf{w}_{2}\in\overline
{g}\overline{R}_{0}$, then the angle between segments $\overline
{\mathbf{v}_{0}\mathbf{w}_{1}}$ and $\overline{\mathbf{v}_{0}\mathbf{w}_{2}}$
is less than $\epsilon$.
\end{corollary}

\begin{proof}
Let $\epsilon>0$ be fixed, and apply Lemma
\ref{Theorem: Scaling Function Property} to obtain $M_{\epsilon/2}$ so large
that if $\overline{g}\overline{R}_{0}$ is a $\overline{BS(m,n)}$-tile in the
preferred positive sheet $%
%TCIMACRO{\U{211d} }%
%BeginExpansion
\mathbb{R}
%EndExpansion
\times\tau^{+}$ and lying outside the closed $M_{\epsilon/2}$-ball of $%
%TCIMACRO{\U{211d} }%
%BeginExpansion
\mathbb{R}
%EndExpansion
\times T\left(  \left\vert m\right\vert ,n\right)  $ centered at
$\mathbf{v}_{0}$, and if $\mathbf{w}_{1},\mathbf{w}_{2}\in\overline
{g}\overline{R}_{0}$, then the angle between $\overline{\mathbf{v}%
_{0}\mathbf{w}_{1}}$ and $\overline{\mathbf{v}_{0}\mathbf{w}_{2}}$ is less
than $\epsilon/2$. Then let $N_{\epsilon}=M_{\epsilon/2}+R$, where $R>0$ is
chosen so large that every $\overline{BS(m,n)}$-tile that intersects $B\left(
\mathbf{v}_{0},M_{\epsilon/2}\right)  $ is contained in $B\left(
\mathbf{v}_{0},M_{\epsilon/2}+R\right)  $.

Now let $\overline{g}\overline{R}_{0}$ be an arbitrary $\overline{BS(m,n)}%
$-tile and $%
%TCIMACRO{\U{211d} }%
%BeginExpansion
\mathbb{R}
%EndExpansion
\times\rho$ a sheet of $T\left(  \left\vert m\right\vert ,n\right)  $
containing $\overline{g}\overline{R}_{0}$.\medskip

\noindent\underline{Case 1: }$%
%TCIMACRO{\U{211d} }%
%BeginExpansion
\mathbb{R}
%EndExpansion
\times\rho$ \emph{is a positive sheet.}\medskip

In the case of the standard tiling of $%
%TCIMACRO{\U{211d} }%
%BeginExpansion
\mathbb{R}
%EndExpansion
\times T\left(  \left\vert m\right\vert ,n\right)  $ (by exponentially growing
rectangles) we observed that the standard tiling of $%
%TCIMACRO{\U{211d} }%
%BeginExpansion
\mathbb{R}
%EndExpansion
\times\rho$ is identical up to horizontal translation to that of $%
%TCIMACRO{\U{211d} }%
%BeginExpansion
\mathbb{R}
%EndExpansion
\times\tau^{+}$. So if the standardly tiled template of $%
%TCIMACRO{\U{211d} }%
%BeginExpansion
\mathbb{R}
%EndExpansion
\times\tau^{+}$ were superimposed on $%
%TCIMACRO{\U{211d} }%
%BeginExpansion
\mathbb{R}
%EndExpansion
\times\rho$, each tile of $%
%TCIMACRO{\U{211d} }%
%BeginExpansion
\mathbb{R}
%EndExpansion
\times\rho$ would be contained in a pair of side-by-side tiles of $%
%TCIMACRO{\U{211d} }%
%BeginExpansion
\mathbb{R}
%EndExpansion
\times\tau^{+}$. This remains true after conjugating the action by $f$.
Therefore the tile $\overline{g}\overline{R}_{0}$ fits within a pair of
side-by-side $\overline{BS\left(  m,n\right)  }$-tiles of $%
%TCIMACRO{\U{211d} }%
%BeginExpansion
\mathbb{R}
%EndExpansion
\times\tau^{+}$ superimposed upon $%
%TCIMACRO{\U{211d} }%
%BeginExpansion
\mathbb{R}
%EndExpansion
\times\rho$. So by the triangle inequality for angle measure and the choice of
$N_{\epsilon}$, the angle between $\overline{\mathbf{v}_{0}\mathbf{w}_{1}}$
and $\overline{\mathbf{v}_{0}\mathbf{w}_{2}}$ is less than $\epsilon$,
provided $\overline{g}\overline{R}_{0}$ lies outside the closed $N_{\epsilon}%
$-ball.\medskip

\noindent\underline{Case 2:} $%
%TCIMACRO{\U{211d} }%
%BeginExpansion
\mathbb{R}
%EndExpansion
\times\rho$ \emph{is arbitrary.}\medskip

As noted previously, the standard tiling of an arbitrary sheet of $%
%TCIMACRO{\U{211d} }%
%BeginExpansion
\mathbb{R}
%EndExpansion
\times T\left(  \left\vert m\right\vert ,n\right)  $ refines the standard
tiling of an appropriately chosen positive sheet. The same then is true for
the $\overline{BS(m,n)}$-tiling. Hence, the general case can be deduced from
Case 1.
\end{proof}

\begin{theorem}
The $\overline{BS(m,n)}$-action on $%
%TCIMACRO{\U{211d} }%
%BeginExpansion
\mathbb{R}
%EndExpansion
\times T\left(  \left\vert m\right\vert ,n\right)  $, together with the visual
compactification $\overline{%
%TCIMACRO{\U{211d} }%
%BeginExpansion
\mathbb{R}
%EndExpansion
\times T\left(  \left\vert m\right\vert ,n\right)  }$ of $%
%TCIMACRO{\U{211d} }%
%BeginExpansion
\mathbb{R}
%EndExpansion
\times T\left(  \left\vert m\right\vert ,n\right)  $ with the $\ell^{2}$
metric, is a $\mathcal{Z}$-structure for $BS\left(  m,n\right)  $.
\end{theorem}

\begin{proof}
We need only verify the nullity condition of Definition
\ref{Definition: Z-structure}. Toward that end let $\mathcal{U}$ be an open
cover of $\overline{%
%TCIMACRO{\U{211d} }%
%BeginExpansion
\mathbb{R}
%EndExpansion
\times T\left(  \left\vert m\right\vert ,n\right)  }$, and apply Lemma
\ref{Lemma: Cover of Boundary} to obtain a $\delta>0$ with the property that
every basic open subset of $\overline{%
%TCIMACRO{\U{211d} }%
%BeginExpansion
\mathbb{R}
%EndExpansion
\times T\left(  \left\vert m\right\vert ,n\right)  }$ of the form $U\left(
z,1/\delta,\delta\right)  $, with $z\in\partial_{\infty}(%
%TCIMACRO{\U{211d} }%
%BeginExpansion
\mathbb{R}
%EndExpansion
\times T\left(  \left\vert m\right\vert ,n\right)  )$, is contained in some
element of $\mathcal{U}$. By Lemma \ref{Lemma: Nullity Condition Check} and
properness of the action, it then suffices to find $N>0$ so that every
$\overline{BS(m,n)}$-translate $\overline{g}\overline{R}_{0}$ of $\overline
{R}_{0}$ which lies outside $\overline{B\left(  \mathbf{v}_{0},N\right)  }$ is
contained in $U\left(  z,\frac{1}{\delta},\delta\right)  $ for some
$z\in\partial_{\infty}(%
%TCIMACRO{\U{211d} }%
%BeginExpansion
\mathbb{R}
%EndExpansion
\times T\left(  \left\vert m\right\vert ,n\right)  )$.

Suppose $\overline{g}\overline{R}_{0}$ lies outside $\overline{B\left(
\mathbf{v}_{0},N\right)  }$, where $N$ is yet to be specified. Choose a sheet
$%
%TCIMACRO{\U{211d} }%
%BeginExpansion
\mathbb{R}
%EndExpansion
\times\rho$ containing $\overline{g}\overline{R}_{0}$ and a point
$\mathbf{w}_{0}\mathbf{\in}\overline{g}\overline{R}_{0}$. The Euclidean ray
$\overrightarrow{\mathbf{v}_{0}\mathbf{w}_{0}}$ in $%
%TCIMACRO{\U{211d} }%
%BeginExpansion
\mathbb{R}
%EndExpansion
\times\rho$ is an element of $\partial_{\infty}(%
%TCIMACRO{\U{211d} }%
%BeginExpansion
\mathbb{R}
%EndExpansion
\times T\left(  \left\vert m\right\vert ,n\right)  )$; call it $z$. Its
projection onto the $\left(  1/\delta\right)  $-sphere of $%
%TCIMACRO{\U{211d} }%
%BeginExpansion
\mathbb{R}
%EndExpansion
\times T\left(  \left\vert m\right\vert ,n\right)  $ is the point $z\left(
1/\delta\right)  $ where the ray $\overrightarrow{\mathbf{v}_{0}\mathbf{w}}$
intersects the semicircle of radius $1/\delta$ in $%
%TCIMACRO{\U{211d} }%
%BeginExpansion
\mathbb{R}
%EndExpansion
\times\rho$. For any other point $\mathbf{w\in}\overline{g}\overline{R}_{0}$
let $p_{1/\delta}\left(  \mathbf{w}\right)  $ denote the projection onto the
$\left(  1/\delta\right)  $-sphere. By the law of cosines, the distance
between $p_{1/\delta}\left(  \mathbf{w}\right)  $ and $z\left(  1/\delta
\right)  $ is $\sqrt{(2/\delta^{2})\left(  1-\cos\left(  \angle\mathbf{w}%
_{0}\mathbf{v}_{0}\mathbf{w}\right)  \right)  }$. Since $\delta$ is constant,
this distance can be made arbitrarily small (in particular $<\delta$), by
forcing $\angle\mathbf{w}_{0}\mathbf{v}_{0}\mathbf{w}$ to be small. By
Corollary \ref{Corollary: the angle lemma}, this can be arranged by making $N$
sufficiently large. Lastly, one should be sure to choose $N>1/\delta$.
\end{proof}

\begin{corollary}
Every generalized Baumslag-Solitar group admits a $\mathcal{Z}$-structure.
\end{corollary}

\begin{proof}
This argument was provided in Section
\ref{Subsection: Generalized Baumslag-Solitar groups}.
\end{proof}

\section{$\mathcal{EZ}$-Structures on Baumslag-Solitar Groups}

We complete the proof of Theorem \ref{Theorem: Boundary BS(m,n)} by showing
that the $\overline{BS(m,n)}$-action on $%
%TCIMACRO{\U{211d} }%
%BeginExpansion
\mathbb{R}
%EndExpansion
\times T\left(  \left\vert m\right\vert ,n\right)  )$ extends to the visual
compactification $\overline{%
%TCIMACRO{\U{211d} }%
%BeginExpansion
\mathbb{R}
%EndExpansion
\times T\left(  \left\vert m\right\vert ,n\right)  )}$. Since this action is
not by isometries and, more specifically, this action does not send rays to
rays, this observation is not immediate.

Note that, since $T\left(  \left\vert m\right\vert ,n\right)  $ is a
Bass-Serre tree for $BS\left(  m,n\right)  $, there is a natural action by
isometries of $BS\left(  m,n\right)  $ on $T\left(  \left\vert m\right\vert
,n\right)  )$. As such, this action extends to the visual compactification of
$T\left(  \left\vert m\right\vert ,n\right)  )$ (which is just its end-point
compactification) in the obvious way. As noted previously, $\partial_{\infty}(%
%TCIMACRO{\U{211d} }%
%BeginExpansion
\mathbb{R}
%EndExpansion
\times T\left(  \left\vert m\right\vert ,n\right)  )$ is the suspension
$S^{0}\ast\partial_{\infty}T\left(  \left\vert m\right\vert ,n\right)  $,
which we may parameterize as the quotient space $[0,\pi]\times\partial
_{\infty}T\left(  \left\vert m\right\vert ,n\right)  /\sim$. Here the
equivalence relation identifies the sets $\left\{  0\right\}  \times
\partial_{\infty}T\left(  \left\vert m\right\vert ,n\right)  $ and $\left\{
\pi\right\}  \times\partial_{\infty}T\left(  \left\vert m\right\vert
,n\right)  $ to the right- and left-hand \emph{suspension points}, which we
denote $\mathbf{R}$ and $\mathbf{L}$. Each edge path ray $\rho$ in $T\left(
\left\vert m\right\vert ,n\right)  $ emanating from $v_{0}$ uniquely
determines both a point of $\partial_{\infty}T\left(  \left\vert m\right\vert
,n\right)  $ and a sheet $%
%TCIMACRO{\U{211d} }%
%BeginExpansion
\mathbb{R}
%EndExpansion
\times\rho\subseteq%
%TCIMACRO{\U{211d} }%
%BeginExpansion
\mathbb{R}
%EndExpansion
\times T\left(  \left\vert m\right\vert ,n\right)  )$. The great semicircle
$C_{\rho}$ of rays in $%
%TCIMACRO{\U{211d} }%
%BeginExpansion
\mathbb{R}
%EndExpansion
\times\rho$ based at $\mathbf{v}_{0}$ (parameterized by the angles they make
with the positive $x$-axis), trace out the set $[0,\pi]\times\{\rho\}\subseteq
S^{0}\ast\partial_{\infty}T\left(  \left\vert m\right\vert ,n\right)  $.

Given a homeomorphism $h:\partial_{\infty}T\left(  \left\vert m\right\vert
,n\right)  \rightarrow\partial_{\infty}T\left(  \left\vert m\right\vert
,n\right)  $, the \emph{suspension of }$h$ is the homeomorphism of $S^{0}%
\ast\partial_{\infty}T\left(  \left\vert m\right\vert ,n\right)  $ which fixes
$\mathbf{R}$ and $\mathbf{L}$ and takes each great semicircle $C_{\rho}$ to
$C_{h\left(  \rho\right)  }$ in a parameter-preserving manner. The
\emph{reflected suspension of }$h$ switches $\mathbf{R}$ and $\mathbf{L}$ and
takes the point on $C_{\rho}$ with parameter $\theta$ to the point on
$C_{h\left(  \rho\right)  }$ with parameter $\pi-\theta$. We will complete the
proof of Theorem \ref{Theorem: Boundary BS(m,n)} for cases $m>0$, by proving
the following proposition.

\begin{proposition}
\label{Proposition: Extension of action to boundary (m>0)}For $m>0$, the
suspension of the $BS(m,n)$-action on $\partial_{\infty}T\left(  \left\vert
m\right\vert ,n\right)  $ extends the $\overline{BS(m,n)}$-action on $%
%TCIMACRO{\U{211d} }%
%BeginExpansion
\mathbb{R}
%EndExpansion
\times T\left(  \left\vert m\right\vert ,n\right)  )$.
\end{proposition}

\begin{remark}
Cases where $m<0$ require the use of reflected suspensions; we will handle
those cases after completing Proposition
\ref{Proposition: Extension of action to boundary (m>0)}.
\end{remark}

Proof of Proposition \ref{Proposition: Extension of action to boundary (m>0)}
requires some additional terminology and notation. Thus far we have understood
the space $%
%TCIMACRO{\U{211d} }%
%BeginExpansion
\mathbb{R}
%EndExpansion
\times T\left(  \left\vert m\right\vert ,n\right)  )$ as a union of sheets,
each with a common origin $\mathbf{v}_{0}=\left(  0,v_{0}\right)  $ and a
common \textquotedblleft edge\textquotedblright, $%
%TCIMACRO{\U{211d} }%
%BeginExpansion
\mathbb{R}
%EndExpansion
\times\left\{  v_{0}\right\}  $. As such, each sheet has a natural system of
Euclidean \emph{local coordinates}, where a point $\left(  x,y\right)  \in%
%TCIMACRO{\U{211d} }%
%BeginExpansion
\mathbb{R}
%EndExpansion
\times\rho$ is represented by the pair of real numbers $\left(  x,d\right)  $,
where $d$ is the distance along $\rho$ from $v_{0}$ to $y$.

Since the actions of $B\left(  m,n\right)  $ on $%
%TCIMACRO{\U{211d} }%
%BeginExpansion
\mathbb{R}
%EndExpansion
\times T\left(  \left\vert m\right\vert ,n\right)  )$ (standard and
conjugated) do not send sheets to sheets, it is useful to expand our
perspective slightly. If $\sigma$ is an arbitrary edge path ray in $T\left(
\left\vert m\right\vert ,n\right)  $ emanating from a vertex $v$, then $%
%TCIMACRO{\U{211d} }%
%BeginExpansion
\mathbb{R}
%EndExpansion
\times\sigma$ is again convex and isometric to a Euclidean half-plane. Call $%
%TCIMACRO{\U{211d} }%
%BeginExpansion
\mathbb{R}
%EndExpansion
\times\sigma$ a \emph{generalized sheet} and attach to it the obvious system
of Euclidean local coordinates, where $\mathbf{v=}\left(  0,v\right)  $ plays
the role of the origin. Note that:

\begin{itemize}
\item if $v_{0}$ lies on $\sigma$, then $%
%TCIMACRO{\U{211d} }%
%BeginExpansion
\mathbb{R}
%EndExpansion
\times\sigma$ contains the sheet $%
%TCIMACRO{\U{211d} }%
%BeginExpansion
\mathbb{R}
%EndExpansion
\times\sigma^{\prime}$ where $\sigma^{\prime}\subseteq\sigma$ is the subray
beginning at $v_{0}$; and

\item if $v_{0}\notin\sigma$, there is an edge path ray $\sigma^{\prime}$
emanating from $v_{0}$ and containing $\sigma$ as a subray, in which case the
sheet $%
%TCIMACRO{\U{211d} }%
%BeginExpansion
\mathbb{R}
%EndExpansion
\times\sigma^{\prime}$ contains $%
%TCIMACRO{\U{211d} }%
%BeginExpansion
\mathbb{R}
%EndExpansion
\times\sigma$.
\end{itemize}

\noindent In each of the above cases, the edges of half-planes $%
%TCIMACRO{\U{211d} }%
%BeginExpansion
\mathbb{R}
%EndExpansion
\times\sigma$ and $%
%TCIMACRO{\U{211d} }%
%BeginExpansion
\mathbb{R}
%EndExpansion
\times\sigma^{\prime}$ cobound a Euclidean strip in the larger of the two
sets. As a result, a ray in $%
%TCIMACRO{\U{211d} }%
%BeginExpansion
\mathbb{R}
%EndExpansion
\times\sigma$ emanating from an arbitrary edge point $\left(  x,v\right)  $ at
an angle $\theta$ with $[x,\infty)\times\left\{  v\right\}  $ is asymptotic in
$%
%TCIMACRO{\U{211d} }%
%BeginExpansion
\mathbb{R}
%EndExpansion
\times T\left(  \left\vert m\right\vert ,n\right)  )$ to the ray in $%
%TCIMACRO{\U{211d} }%
%BeginExpansion
\mathbb{R}
%EndExpansion
\times\sigma^{\prime}$ emanating from $\mathbf{v}_{0}$ and forming the same
angle with $[0,\infty)\times v_{0}$. As such both rays represent the same
element of $S^{0}\ast\partial_{\infty}T\left(  m,n\right)  $, the point on the
the semicircle $C_{\sigma^{\prime}}$ with parameter $\theta$.

\begin{proof}
[Proof of Proposition \ref{Proposition: Extension of action to boundary (m>0)}
]In this proof we allow $s$ and $t$ to represent the isometries generating the
action of $BS(m,n)$ on the Bass-Serre tree $T\left(  \left\vert m\right\vert
,n\right)  )$ as well as the extensions of those isometries to the visual
compactification of $T\left(  \left\vert m\right\vert ,n\right)  )$. We use
the same symbols to denote the homeomorphisms generating the standard
$BS\left(  m,n\right)  $-action on $%
%TCIMACRO{\U{211d} }%
%BeginExpansion
\mathbb{R}
%EndExpansion
\times T\left(  \left\vert m\right\vert ,n\right)  )$, as described in Section
\ref{Subsection: standard action}$\footnote{This notation is reasonable since
the isometries $s,t:T\left(  m,n\right)  \rightarrow T\left(  m,n\right)  $
are precisely the $T\left(  m,n\right)  $-coordinate functions of the
corresponding self-homeomorphisms of $%
%TCIMACRO{\U{211d} }%
%BeginExpansion
\mathbb{R}
%EndExpansion
\times T\left(  m,n\right)  $.}$. It will be useful to have formulaic
representations of these functions.

As an isometry of $T\left(  \left\vert m\right\vert ,n\right)  )$, $s$ fixes
$v_{0}$, but permutes the collection of rays emanating from that vertex. As a
self-homeomorphism of $%
%TCIMACRO{\U{211d} }%
%BeginExpansion
\mathbb{R}
%EndExpansion
\times T\left(  \left\vert m\right\vert ,n\right)  )$, the action of $s$ on
the $%
%TCIMACRO{\U{211d} }%
%BeginExpansion
\mathbb{R}
%EndExpansion
$-coordinate is translation by $1/n$. So, if $%
%TCIMACRO{\U{211d} }%
%BeginExpansion
\mathbb{R}
%EndExpansion
\times\rho$ is an arbitrary sheet and $\rho^{\prime}$ is the image of $\rho$
under $s$ in the Bass-Serre tree, then, as a homeomorphism of $%
%TCIMACRO{\U{211d} }%
%BeginExpansion
\mathbb{R}
%EndExpansion
\times T\left(  \left\vert m\right\vert ,n\right)  )$, $s$ takes points of $%
%TCIMACRO{\U{211d} }%
%BeginExpansion
\mathbb{R}
%EndExpansion
\times\rho$ with local coordinates $\left(  x,d\right)  $ to points of $%
%TCIMACRO{\U{211d} }%
%BeginExpansion
\mathbb{R}
%EndExpansion
\times\rho^{\prime}$ with local coordinates $\left(  x+\frac{1}{n},d\right)  $.

As an isometry of $T\left(  \left\vert m\right\vert ,n\right)  )$, $t$ sends
$v_{0}$ to a vertex $v_{1}$, one unit away; and as a self-homeomorphism of $%
%TCIMACRO{\U{211d} }%
%BeginExpansion
\mathbb{R}
%EndExpansion
\times T\left(  \left\vert m\right\vert ,n\right)  )$, the action of $t$ on
the $%
%TCIMACRO{\U{211d} }%
%BeginExpansion
\mathbb{R}
%EndExpansion
$-coordinate is multiplication by $n/m$. So, if $%
%TCIMACRO{\U{211d} }%
%BeginExpansion
\mathbb{R}
%EndExpansion
\times\rho$ is an arbitrary sheet and $\rho^{\prime}$ is its image under $t$
in the Bass-Serre tree, $s$ takes points of $%
%TCIMACRO{\U{211d} }%
%BeginExpansion
\mathbb{R}
%EndExpansion
\times\rho$ with local coordinates $\left(  x,d\right)  $ to points of $%
%TCIMACRO{\U{211d} }%
%BeginExpansion
\mathbb{R}
%EndExpansion
\times\rho^{\prime}$ with local coordinates $\left(  \left(  \frac{n}%
{m}\right)  x,d\right)  $.

Now consider the homeomorphisms $\overline{s}=f\circ s\circ f^{-1}$ and
$\overline{t}=f\circ t\circ f^{-1}$ which generate the $\overline{B\left(
m,n\right)  }$-action on $%
%TCIMACRO{\U{211d} }%
%BeginExpansion
\mathbb{R}
%EndExpansion
\times T\left(  \left\vert m\right\vert ,n\right)  )$. Since the suspension of
a composition is the composition of the suspensions, it is enough to verify
the proposition for these two elements. Recall that $f$ and $f^{-1}$ are given
by the formulas%
\begin{align*}
&  \left(  x,y\right)  \overset{f}{\longrightarrow}\left(  \operatorname{sgn}%
\left(  x\right)  \log\left(  \log\left(  \left\vert x\right\vert +e\right)
\right)  ,y\right)  \text{,\hspace{0.15in}and}\\
&  \left(  x,y\right)  \overset{f^{-1}}{\longrightarrow}\left(
\operatorname{sgn}\left(  x\right)  (\exp\left(  \exp\left(  \left\vert
x\right\vert \right)  \right)  -e\right)  ,y)
\end{align*}

Let $%
%TCIMACRO{\U{211d} }%
%BeginExpansion
\mathbb{R}
%EndExpansion
\times\rho$ be an arbitrary sheet, and for $p,q\in%
%TCIMACRO{\U{2124} }%
%BeginExpansion
\mathbb{Z}
%EndExpansion
$ with $p\geq0$, let $\overrightarrow{\mathbf{r}_{\frac{p}{q}}}=\left\{
\left(  qx,px\right)  \mid x\in%
%TCIMACRO{\U{211d} }%
%BeginExpansion
\mathbb{R}
%EndExpansion
^{+}\right\}  $, i.e., $\overrightarrow{\mathbf{r}_{\frac{p}{q}}}$ is the ray
in $%
%TCIMACRO{\U{211d} }%
%BeginExpansion
\mathbb{R}
%EndExpansion
\times\rho$ with slope $p/q$. If $\rho^{\prime}$ is the image of $\rho$ under
$s$ in the Bass-Serre tree, then $\overline{s}$ takes $%
%TCIMACRO{\U{211d} }%
%BeginExpansion
\mathbb{R}
%EndExpansion
\times\rho$ onto $%
%TCIMACRO{\U{211d} }%
%BeginExpansion
\mathbb{R}
%EndExpansion
\times\rho^{\prime}$ and the image of $\overrightarrow{\mathbf{r}_{\frac{p}%
{q}}}$ is the set of points with local coordinates
\begin{equation}
\{(\delta_{q,x,n}\cdot\log(\log(\exp\left(  \exp\left(  \left\vert
q\right\vert x\right)  \right)  +\frac{1}{n})),px)\mid x\in%
%TCIMACRO{\U{211d} }%
%BeginExpansion
\mathbb{R}
%EndExpansion
^{+}\},\label{formula: topological ray 1}%
\end{equation}
where $\delta_{q,x,n}=\pm1$ is a small variation on $\operatorname{sgn}\left(
q\right)  $. Specifically,
\[
\delta_{q,x,n}=\operatorname{sgn}(\operatorname{sgn}\left(  q\right)
\log(\log(x+e))+\frac{1}{n})
\]
which is identical to $\operatorname{sgn}\left(  q\right)  $ except when
$\log\left(  \log\left(  x+e\right)  \right)  <\frac{1}{n}$ and $q<0$. Most
importantly, the image of $\overrightarrow{\mathbf{r}_{\frac{p}{q}}}$ under
$\overline{s}$ is a topologically embedded (non-geodesic) ray in $%
%TCIMACRO{\U{211d} }%
%BeginExpansion
\mathbb{R}
%EndExpansion
\times\rho^{\prime}$ which, in local coordinates, emanates from $\left(
\log\log\left(  e+\frac{1}{n}\right)  ,0\right)  $ and is asymptotic to
geodesic rays in $%
%TCIMACRO{\U{211d} }%
%BeginExpansion
\mathbb{R}
%EndExpansion
\times\rho^{\prime}$ with slope $p/q$. That is easily seen by letting $x$
approach infinity in formula (\ref{formula: topological ray 1}). From this it
can be seen that the restriction of $\overline{s}$ taking $%
%TCIMACRO{\U{211d} }%
%BeginExpansion
\mathbb{R}
%EndExpansion
\times\rho$ onto $%
%TCIMACRO{\U{211d} }%
%BeginExpansion
\mathbb{R}
%EndExpansion
\times\rho^{\prime}$ extends to the visual boundaries of these half-planes by
taking $C_{\rho}$ onto $C_{\rho^{\prime}}$ in a parameter preserving manner.
Since this is true for each sheet, it follows that the suspension of the
homeomorphism $s:\partial_{\infty}T\left(  \left\vert m\right\vert ,n\right)
\rightarrow\partial_{\infty}T\left(  \left\vert m\right\vert ,n\right)  $
extends $\overline{s}:%
%TCIMACRO{\U{211d} }%
%BeginExpansion
\mathbb{R}
%EndExpansion
\times T\left(  \left\vert m\right\vert ,n\right)  )\rightarrow%
%TCIMACRO{\U{211d} }%
%BeginExpansion
\mathbb{R}
%EndExpansion
\times T\left(  \left\vert m\right\vert ,n\right)  )$ over the visual boundary.

Next consider the homeomorphism $\overline{t}$. Again let
$\overrightarrow{\mathbf{r}_{\frac{p}{q}}}$ be a ray (as described above) in
an arbitrary sheet $%
%TCIMACRO{\U{211d} }%
%BeginExpansion
\mathbb{R}
%EndExpansion
\times\rho$ and let $\rho^{\prime}$ be the $t$-image of $\rho$ under the
action on $T\left(  \left\vert m\right\vert ,n\right)  $. In local
coordinates, the image of $\overrightarrow{\mathbf{r}_{\frac{p}{q}}}$ is the
set of points in $%
%TCIMACRO{\U{211d} }%
%BeginExpansion
\mathbb{R}
%EndExpansion
\times\rho^{\prime}$ with local coordinates
\[
\{(\operatorname{sgn}\left(  q\right)  \log\left(  \log\left(  \frac{n}{m}%
\exp\left(  \exp\left(  \left\vert q\right\vert x\right)  \right)
+(\frac{m-n}{m})\cdot e\right)  \right)  ,px)\mid x\in%
%TCIMACRO{\U{211d} }%
%BeginExpansion
\mathbb{R}
%EndExpansion
^{+}\}.
\]

Consider now the ratios of the coordinates of these points as $x$ gets large,
i.e.,
\[
\operatorname{sgn}\left(  q\right)  \cdot\lim_{x\rightarrow\infty}\frac
{px}{\log\left(  \log\left(  \frac{n}{m}\exp\left(  \exp\left(  \left\vert
q\right\vert x\right)  \right)  +(\frac{m-n}{m})\cdot e\right)  \right)  }%
\]
By another elementary but messy calculation involving L'H\^{o}pital's Rule,
this limit is $p/q$. As such, the image of $\overrightarrow{\mathbf{r}%
_{\frac{p}{q}}}$ under $\overline{t}$ is a topologically embedded
(non-geodesic) ray in $%
%TCIMACRO{\U{211d} }%
%BeginExpansion
\mathbb{R}
%EndExpansion
\times\rho^{\prime}$ emanating (in local coordinates) from $\left(
0,0\right)  $ and asymptotic to rays in $%
%TCIMACRO{\U{211d} }%
%BeginExpansion
\mathbb{R}
%EndExpansion
\times\rho^{\prime}$ with slope $p/q$. As before, the restriction of
$\overline{t}$ taking $%
%TCIMACRO{\U{211d} }%
%BeginExpansion
\mathbb{R}
%EndExpansion
\times\rho$ onto $%
%TCIMACRO{\U{211d} }%
%BeginExpansion
\mathbb{R}
%EndExpansion
\times\rho^{\prime}$ extends to the visual boundaries of these half-planes by
taking $C_{\rho}$ onto $C_{\rho^{\prime}}$ in a parameter preserving manner.
And since this is true for all sheets, the suspension of $t:\partial_{\infty
}T\left(  \left\vert m\right\vert ,n\right)  \rightarrow\partial_{\infty
}T\left(  \left\vert m\right\vert ,n\right)  $ extends $\overline{t}:%
%TCIMACRO{\U{211d} }%
%BeginExpansion
\mathbb{R}
%EndExpansion
\times T\left(  \left\vert m\right\vert ,n\right)  )\rightarrow%
%TCIMACRO{\U{211d} }%
%BeginExpansion
\mathbb{R}
%EndExpansion
\times T\left(  \left\vert m\right\vert ,n\right)  )$ over the visual boundary.
\end{proof}

To complete Theorem \ref{Theorem: Boundary BS(m,n)}, we need an analog of
Proposition \ref{Proposition: Extension of action to boundary (m>0)} for
$m<0$. In those cases, we cannot simply suspend the $BS\left(  m,n\right)
$-action on $\partial_{\infty}T\left(  \left\vert m\right\vert ,n\right)  $ to
get the appropriate extension of the $\overline{BS(m,n)}$-action on $%
%TCIMACRO{\U{211d} }%
%BeginExpansion
\mathbb{R}
%EndExpansion
\times T\left(  \left\vert m\right\vert ,n\right)  )$. That is because
homeomorphisms $t$ and $\overline{t}$ now flip the orientation of the $%
%TCIMACRO{\U{211d} }%
%BeginExpansion
\mathbb{R}
%EndExpansion
$-factor. More precisely, if $r:%
%TCIMACRO{\U{211d} }%
%BeginExpansion
\mathbb{R}
%EndExpansion
\times T\left(  \left\vert m\right\vert ,n\right)  )\rightarrow%
%TCIMACRO{\U{211d} }%
%BeginExpansion
\mathbb{R}
%EndExpansion
\times T\left(  \left\vert m\right\vert ,n\right)  )$ is the reflection
homeomorphism taking $\left(  x,y\right)  $ to $\left(  -x,y\right)  $, then
$t$ and $\overline{t}$ are the homeomorphisms $r\circ t^{\prime}$ and
$r\circ\overline{t^{\prime}}$, where $t^{\prime}$ and $\overline{t^{\prime}}$
are the homeomorphisms studied earlier in cases where $m>0$. Obviously, if
$\overline{t^{\prime}}$ extends to the visual boundary of $%
%TCIMACRO{\U{211d} }%
%BeginExpansion
\mathbb{R}
%EndExpansion
\times T\left(  \left\vert m\right\vert ,n\right)  )$ by suspending the
corresponding homeomorphism of $\partial_{\infty}T\left(  \left\vert
m\right\vert ,n\right)  )$, then $\overline{t}$ extends to the visual boundary
of $%
%TCIMACRO{\U{211d} }%
%BeginExpansion
\mathbb{R}
%EndExpansion
\times T\left(  \left\vert m\right\vert ,n\right)  )$ via the reflected
suspension of that same homeomorphism. By contrast, the homeomorphisms $s$ and
$\overline{s}$ are no different when $m<0$ than they are when $m>0$.

For $m<0$ define $\phi:BS\left(  m,n\right)  \rightarrow%
%TCIMACRO{\U{2124} }%
%BeginExpansion
\mathbb{Z}
%EndExpansion
$ to be the quotient map obtained by modding out by the normal closure of the
subgroup $\left\langle s\right\rangle $. Then, for an action of $BS\left(
m,n\right)  $ on $\partial_{\infty}T\left(  \left\vert m\right\vert ,n\right)
)$, define the corresponding $t$\emph{-reflected action} of $BS\left(
m,n\right)  $ on $S^{0}\ast\partial_{\infty}T\left(  \left\vert m\right\vert
,n\right)  )$ as follows:

\begin{itemize}
\item if $\phi\left(  g\right)  $ is even, then $g:S^{0}\ast\partial_{\infty
}T\left(  \left\vert m\right\vert ,n\right)  )\rightarrow S^{0}\ast
\partial_{\infty}T\left(  \left\vert m\right\vert ,n\right)  )$ is the
suspension of $g:\partial_{\infty}T\left(  \left\vert m\right\vert ,n\right)
)\rightarrow\partial_{\infty}T\left(  \left\vert m\right\vert ,n\right)  )$, and

\item if $\phi\left(  g\right)  $ is odd, then $g:S^{0}\ast\partial_{\infty
}T\left(  \left\vert m\right\vert ,n\right)  )\rightarrow S^{0}\ast
\partial_{\infty}T\left(  \left\vert m\right\vert ,n\right)  )$ is the
reflected suspension of $g:\partial_{\infty}T\left(  \left\vert m\right\vert
,n\right)  )\rightarrow\partial_{\infty}T\left(  \left\vert m\right\vert
,n\right)  )$.
\end{itemize}

\noindent Proof of the following is now essentially the same as Proposition
\ref{Proposition: Extension of action to boundary (m>0)}.

\begin{proposition}
For $m<0$, the $t$-reflected suspension of the $BS(m,n)$-action on
$\partial_{\infty}T\left(  \left\vert m\right\vert ,n\right)  $ extends the
$\overline{BS(m,n)}$-action on $%
%TCIMACRO{\U{211d} }%
%BeginExpansion
\mathbb{R}
%EndExpansion
\times T\left(  \left\vert m\right\vert ,n\right)  )$.
\end{proposition}

\begin{remark}
The argument by which $\mathcal{Z}$-structures for generalized
Baumslag-Solitar groups were obtained from the existence of $\mathcal{Z}%
$-structures on ordinary Baumslag-Solitar groups does not extend to
$\mathcal{EZ}$-structures. That is because equivariance can be lost when
applying Proposition \ref{Boundary Swapping Theorem}. We leave the issue of
$\mathcal{EZ}$-structures for generalized Baumslag-Solitar groups for later.
\end{remark}

\bibliographystyle{amsalpha}
\bibliography{Biblio}

\end{document}